\documentclass{amsproc}
\usepackage{tikz,amssymb}
\newtheorem{theorem}{Theorem}[section]
\newtheorem{lemma}[theorem]{Lemma}
\theoremstyle{definition}
\newtheorem{definition}[theorem]{Definition}
\newtheorem{example}[theorem]{Example}

\theoremstyle{remark}
\newtheorem{remark}[theorem]{Remark}

\numberwithin{equation}{section}

\newcommand{\myK}{\mathbf{\Delta Gpd}}
\newcommand{\myR}{\mathbf{Ring}}

\begin{document}

\title{Delta-groupoids and ideal triangulations}

\author{R. M.  Kashaev}
\address{Universit\'e de Gen\`eve,
Section de math\'ematiques,
2-4, rue du Li\`evre,
CP 64,
1211 Gen\`eve 4, Suisse}
\email{rinat.kashaev@unige.ch}
\thanks{The work is supported by the Swiss National Science
  Foundation.}


\date{January 17, 2010}



\begin{abstract}
A $\Delta$-groupoid is an algebraic structure which axiomatizes the combinatorics of a truncated tetrahedron. By considering two simplest examples coming from knot theory, we illustrate how can one associate a $\Delta$-groupoid to  an ideal triangulation of a three-manifold. We also describe in detail the rings associated with the $\Delta$-groupoids of these examples.
\end{abstract}

\maketitle
\section{Introduction}
This paper is about the algebraic structure called $\Delta$-groupoid introduced in the papers~\cite{K2,K3}. The motivation for studying this structure comes from the combinatorial descriptions of Teichm\"uller spaces of punctured surfaces \cite{Thurston,Penner,Bonahon} which were useful for quantization \cite{Fock-Chekhov,K1}, and from the fact that it permits to describe essential parts of the set of representations of the fundamental group of three-manifolds into matrix groups over arbitrary rings \cite{K3}.

In this paper, we relate the definition of the $\Delta$-groupoid to the combinatorics of truncated tetrahedra, and, by using two examples of ideal triangulations of complements of the trefoil and the figure-eight knots, we illustrate how ideal triangulations lead to presentations of $\Delta$-groupoids. Relations of $\Delta$-groupoids to representations of the fundamental group of three-manifolds into two-by-two matrix groups over arbitrary rings imply that there are at least two pairs of adjoint functors between the categories of (unital) rings and $\Delta$-groupoids. These adjoint functors permit to associate to any $\Delta$-groupoid a pair of rings to be referred as $A'$ and $B'$ rings.  We  work out in detail the structure of these rings in our examples. It should be remarked that the isomorphism classes of these rings are not three-manifold invariants as they depend on the choice of triangulation. In a sense, different ideal triangulations correspond to different local charts in the moduli space of representations of the fundamental group of a manifold into a given matrix group. In analogy with two charts related by a transformation on their common part, the rings, associated with two different triangulations of one and the same three-manifold, admit a common localization.  The situation here is very similar the one with deformation varieties \cite{Thurston,Yoshida,Tillmann}, where to cover the moduli space of $PSL_2(\mathbb{C})$-representations, one may need several ideal triangulations. In the case of $PSL_2(\mathbb{R})$-representations of surface groups, this kind of phenomenon has been demonstrated in \cite{K4}. In general, the $A'$ ring is a quotient ring of the $B'$ ring, so that in principle it is enough to calculate only the $B'$-ring, but in practice  the $A'$ can be simpler to calculate.

In the case of the trefoil knot, the two rings are commutative and both are isomorphic to the ring $\mathbb{Z}[t,3^{-1}]/(\Delta_{3_1}(t))$, where $\Delta_{3_1}(t)=t^2-t+1$ is the Alexander polynomial of the trefoil knot. The case of the figure-eight knot is less trivial and  reveals interesting features: both rings are non-commutative and not isomorphic to each other.  The non-commutativity, however, is rather mild in the sense that the rings are quotients of formalized versions of two-by-two matrix rings over commutative rings. In this example,  the $B'$-ring contains a non-trivial torsion, namely, an element $\epsilon$ satisfying the equation $5\epsilon=0$, and which disappears on the level of the $A'$-ring.

 The organization of the paper is as follows. In Section~\ref{mysec:1}, by studying certain combinatorial features of a truncated tetrahedron, we naturally arrive to the definition of the $\Delta$-groupoid. Then, after giving a list of examples of $\Delta$-groupoids, we define the notion of a labeled truncated tetrahedron, and introduce a graphical notation for it. In Section~\ref{mysec:2} we consider in detail two simplest examples of ideal triangulations (both consisting of only two ideal tetrahedra) of the complements of the  trefoil and the figure-eight knots. In the part dedicated to the figure-eight knot, for technical reasons, we introduce the class of formal $M_2$-rings which contains as a sub-class the rings of two-by-two matrices over commutative rings. The main result of the paper is given by Theorem~\ref{T:main}.

\section{Truncated tetrahedra and $\Delta$-groupoids}\label{mysec:1}
Consider a truncated tetrahedron with its natural cell decomposition drawn in this picture
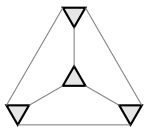
\begin{center}
\begin{tikzpicture}[>=stealth,scale=.3]
\draw[thick,fill=black!10] (0,0)--(60:1)--(1,0)--cycle;\draw[thick,fill=black!10] (-1,0)++(-120:1)--++(-1,0)--+(-60:1)--cycle;
\draw[thick,fill=black!10] (60:2)++(120:1)--++(60:1)--+(-1,0)--cycle;\draw[thick,fill=black!10] (2,0)++(-60:1)--++(1,0)--+(-120:1)--cycle;
\draw[black!50] (-1,0)+(-120:1)--(0,0);\draw[black!50] (2,0)+(-60:1)--(1,0);\draw[black!50] (60:2)+(120:1)--(60:1);
\draw[black!50] (-2,0)++(-120:1)--+(60:5);\draw[black!50] (3,0)++(-60:1)--+(120:5);\draw[black!50] (-120:2)++(-1,0)--+(5,0);
\end{tikzpicture}
\end{center}
This is a three-dimensional cell complex which consists of twelve  $0$-cells (vertices), eighteen  $1$-cells (edges), eight $2$-cells (faces), and one $3$-cell. There are two types of $2$-cells: the triangular faces (those which are shaded in the above picture) and the hexagonal faces which are remnants of the faces of the corresponding non-truncated tetrahedron. Note that there are four $2$-cells of each type. There are also two types of edges: twelve \emph{short} edges (those drawn by thick lines in the picture) which bound the triangular faces and six \emph{long} edges (drawn by thin lines) which are remnants of the edges of the non-truncated tetrahedron.

The set of triangular faces  constitutes a two-dimensional cell sub-complex whose fundamental edge-path groupoid will be denoted $G$. Algebraically, this is not very interesting groupoid: it consists of four connected components (one component for each triangular face) any of which is isomorphic to one and the same coarse groupoid over a three-element set\footnote{A \emph{coarse} groupoid over a set $S$ is the groupoid $S^2$ where the set of objects is $S$ and the set of morphisms (arrows) is the cartesian square of $S$ with the domain and the codomain maps being the projections to the first and second components respectively, and the product $(x,y)(y,z)=(x,z)$.}, but it becomes more interesting due to the following additional structure.

We remark that the set $H\subset G$ of oriented short edges is naturally a free $S_3$-set, where $S_3$ is the permutation group of three elements. Indeed, the elements of $H$ can be grouped into four six-element families associated with hexagonal faces, and a free transitive action of the group $S_3$ within each family is defined through realization of the three elementary transpositions in $S_3$ by three involutions $i,j,k\colon H\to H$, where $i(x)=x^{-1}$ while $j(x)$ and $k(x)$ are the unique oriented short edges such that the domain (respectively the codomain) of $x$ belongs to the same long edge as the domain (respectively the codomain) of $j(x)$ (respectively $k(x)$) as is described in this picture:
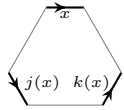
\begin{center}
\begin{tikzpicture}[>=stealth,scale=0.5]
\draw[thick](0,0)--(120:1);\draw[thick](2,0)--+(60:1);\draw[thick](60:3)--+(-1,0);
\draw[black!50](0,0)--(2,0);\draw[black!50](120:1)--+(60:2);\draw[black!50](60:3)--+(-60:2);
\draw[thick,->] (120:.501)--(120:.5);\draw[thick,->] (60:3)+(-.501,0)--+(-.5,0);
\draw[thick,->] (2,0)+(60:.5)--+(60:.501);
\draw (1,2.4) node{\tiny $x$};\draw (.4,.6) node{\tiny $j(x)$};\draw (1.7,.6) node{\tiny $k(x)$};
\end{tikzpicture}
\end{center}
It is easy to check that the defining relations of the group $S_3$ are satisfied, in particular, the equalities
\begin{equation}\label{E:1}
k=iji=jij.
\end{equation}

 As the inversion map $i$ and the product operation are related through the identity $i(xy)=i(y)i(x)$, the following question arises: are there some other identities in the groupoid $G$ involving the involutions $i,j,k$ and the product operation? To answer this question,
 let us say that a pair of elements $(x,y)\in H\times H$ is $H$-\emph{composable} if it is composable in $G$ and $xy\in H$. Not any composable pair is $H$-composable, for example, for any $x\in H$, the pair $(x,x^{-1})$ is composable but not $H$-composable.

 Picking up an arbitrary $H$-composable pair $(x,y)$, we can label all the elements of the set $H$ according to this picture
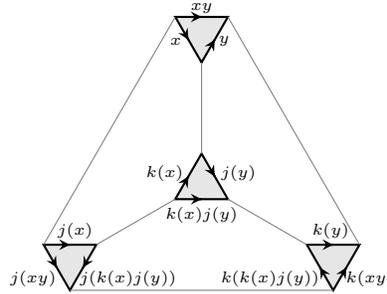
\begin{center}
\begin{tikzpicture}[>=stealth,scale=.7]
\draw[thick,fill=black!10] (0,0)--(60:1)--(1,0)--cycle;\draw[thick,fill=black!10] (-1,0)++(-120:1)--++(-1,0)--+(-60:1)--cycle;
\draw[thick,fill=black!10] (60:2)++(120:1)--++(60:1)--+(-1,0)--cycle;\draw[thick,fill=black!10] (2,0)++(-60:1)--++(1,0)--+(-120:1)--cycle;
\draw[black!50] (-1,0)+(-120:1)--(0,0);\draw[black!50] (2,0)+(-60:1)--(1,0);\draw[black!50] (60:2)+(120:1)--(60:1);
\draw[black!50] (-2,0)++(-120:1)--+(60:5);\draw[black!50] (3,0)++(-60:1)--+(120:5);\draw[black!50] (-120:2)++(-1,0)--+(5,0);
\draw[thick,->] (60:.5)--(60:.501);\draw[thick,->] (60:1)+(-60:.5)--+(-60:.501);
\draw[thick,->] (60:2)+(120:1.501)--+(120:1.5);\draw[thick,->] (120:1)+(60:2.5)--+(60:2.501);
\draw[thick,->] (.5,0)--(.501,0);\draw[thick,->] (120:2)++(60:2)+(.5,0)--+(.501,0);
\draw[thick,->] (-1,0)++(-120:1)+(-.501,0)--+(-.5,0);
\draw[thick,->] (2,0)++(-60:1)+(.5,0)--+(.501,0);
\draw[thick,->] (-2,0)++(-120:1)+(-60:.5)--+(-60:.501);
\draw[thick,->] (-1,0)+(-120:1.5)--+(-120:1.501);
\draw[thick,->] (2,0)+(-60:1.501)--+(-60:1.5);
\draw[thick,->] (2,0)++(-60:2)+(60:.5)--+(60:.501);
\draw (-.2,.5) node{\tiny $k(x)$};\draw (1.2,.5) node{\tiny $j(y)$};\draw (.01,3) node{\tiny $x$};
\draw (.93,3) node{\tiny $y$};\draw (-1.9,-.6) node{\tiny $j(x)$};\draw (2.95,-.6) node{\tiny $k(y)$};
\draw (.5,3.63) node{\tiny $xy$};
\draw (.5,-.3) node{\tiny $k(x)j(y)$};
\draw (-2.7,-1.5) node{\tiny $j(xy)$};\draw (-.9,-1.5) node{\tiny $j(k(x)j(y))$};
\draw (1.8,-1.5) node{\tiny $k(k(x)j(y))$};\draw (3.7,-1.5) node{\tiny $k(xy)$};
\end{tikzpicture}
\end{center}
from where, in particular, we observe that the pair $(k(x),j(y))$ is also $H$-composable and the following identities hold true:
\[
j(xy)=j(x)j(k(x)j(y)),\quad k(xy)=k(k(x)j(y))k(y).
\]
Moreover, taking into account equalities~\eqref{E:1}, the latter two identities are not independent, and in fact one follows from another as shows the following calculation:
\begin{multline*}
k(xy)=iji(xy)=ij(i(y)i(x))=i(ji(y)j(ki(y)ji(x)))\\=ij(ki(y)ji(x))iji(y)=iji(iji(x)iki(y))iji(y)=
k(k(x)j(y))k(y).
\end{multline*}
All of above motivates the following definition~\cite{K3}.
\begin{definition}\label{def:delta}
A \emph{  $\Delta$-groupoid} is a groupoid $G$, a generating subset $H\subset G$, and an involution $j\colon H\to H$, such that
\begin{itemize}
\item[(i)] $i(H)=H$, where $i(x)=x^{-1}$;
\item[(ii)] the involutions $i$ and $j$ generate an action of the symmetric group $S_3$ on the set $H$, i.e. the following equation is satisfied: $iji=jij$;
\item[(iii)] if $(x,y)\in H^2$ is a composable pair then $(k(x),j(y))$ is also a composable pair, where $k=iji$;
\item[(iv)] if $(x,y)\in H^2$ is an $H$-composable pair then $(k(x),j(y))$ is also an $H$-composable pair, and
the following identity is satisfied:
\begin{equation}\label{myeq:1}
j(xy)=j(x)j(k(x)j(y)).
\end{equation}
\end{itemize}
A \emph{$\Delta$-group} is a $\Delta$-groupoid  with one object (identity element).
\end{definition}
In the category $\myK$ of   $\Delta$-groupoids a morphism  between two $\Delta$-groupoids is a groupoid morphism $f\colon G\to G'$ such that $f(H)\subset H'$ and $j'f=fj$.
\begin{remark}\label{rm:inv}
In any $\Delta$-groupoid $G$ there is a canonical involution $A\mapsto A^*$ acting on the set of objects (or the identities) of $G$. It can be defined as follows. Let $\mathrm{dom},\mathrm{cod}\colon G\to \mathrm{Ob}G$ be the domain (source) and the codomain (target) maps, respectively. As $H$ is a generating set for $G$, for any $A\in\mathrm{Ob}G$ there exists $x\in H$ such that $A=\mathrm{dom}(x)$. We define $A^*=\mathrm{dom}(j(x))$. This definition is independent of the choice of $x$. Indeed, let $y\in H$ be any other element satisfying the same condition. Then, the pair $(i(y),x)$  is composable, and, therefore, so is $(ki(y),j(x))$. Thus\footnote{For compositions in groupoids we use the convention adopted for fundamental groupoids of topological spaces, i.e. $(x,y)\in G^2$ is composable iff $\mathrm{cod}(x)=\mathrm{dom}(y)$, and the product is written $xy$ rather than $y\circ x$.},
\[
\mathrm{dom}(j(y))=\mathrm{cod}(ij(y))=\mathrm{cod}(ki(y))=\mathrm{dom}(j(x)).
\]
\end{remark}
\subsection{Examples of $\Delta$-groupoids}
 \begin{example}\label{myex:11}
 Let $G$ be a group. The coarse groupoid $G^{2}$ is a $\Delta$-groupoid with $H=G^2$, $j(f,g)=(f^{-1},f^{-1}g)$.
  \end{example}
 \begin{example}\label{myex:0}
 Let $X$ be a set. The set
 \(
 X^{3 }
 \)
can be thought of as a disjoint sum of coarse groupoids $X^2$ indexed by $X$: $X^3\simeq\sqcup_{x\in X}\{x\}\times X^2$. In particular, $\mathrm{Ob}(X^3)=X^2$ with $\mathrm{dom}(a,b,c)=(a,b)$ and $\mathrm{cod}(a,b,c)=(a,c)$. This is a $\Delta$-groupoid with $H=X^3$ and $j(a,b,c)=(b,a,c)$. In the case where $X$ is a three element set this $\Delta$-groupoid is isomorphic to the one associated with the truncated tetrahedron.
 \end{example}
\begin{example}\label{myex:1}
Let $R$ be a ring. We define a $\Delta$-group $AR$ as the subgroup of the group $R^* $ of invertible elements of $R$ generated by the subset $H=(1-R^* )\cap R^*$ with $k(x)=1-x$.
\end{example}
\begin{example}\label{myex:2}
For a ring $R$, let $R\rtimes R^* $ be the semidirect product of the additive group $R$ with the multiplicative group $R^*$ with respect to the (left) action of $R^* $ on $R$ by left multiplications. Set theoretically, one has $R\rtimes R^* =R\times R^* $, the group structure being given explicitly by the product $(x,y)(u,v)=(x+yu,yv)$,  the unit element $(0,1)$, and the inversion map $(x,y)^{-1}=(-y^{-1}x,y^{-1})$. We define a $\Delta$-group $BR$ as the subgroup of $R\rtimes R^*$ generated by the subset $H=R^* \times R^*$ with $k(x,y)=(y,x)$.
\end{example}
\begin{example}
Let $(G,G_\pm,\theta)$ be a symmetrically factorized group of \cite{KR}. That means that $G$ is a group with two isomorphic subgroups $G_\pm $ conjugated to each other by an involutive element $\theta\in G$, and the restriction of the multiplication  map $m\colon G_+\times G_-\to G_+G_-\subset G$ is a set-theoretical bijection, whose inverse is called the factorization map $G_+G_-\ni g\mapsto (g_+,g_-)\in G_+\times G_-$. In this case, the subgroup of $G_+$ generated by the subset $H=G_+\cap G_-G_+\theta\cap\theta G_+G_-$ is a  $\Delta$-group with $j(x)=(\theta x)_+$.
\end{example}
\begin{example}
Let $A$ be a group and $B\subset A$ a malnormal subgroup.\footnote{A subgroup $B\subset A$ is called \emph{mal-normal} if $aBa^{-1}\cap B=\{1\}$ for any $a\in A\setminus B$.} In this case the (left) action  of the group $B$ on the set of non-trivial left cosets $X=\{ aB\vert\ a\in A\setminus B\}$ is free. Let $G$ be the groupoid associated to this free action\footnote{The groupoid associated to a free $G$-set $X$ has the $G$-orbits in $X$ as objects and the $G$-orbits in $X\times X$ (with respect to the diagonal action) as morphisms, the domain and the codomain maps being respectively $G(x,y)\mapsto Gx$ and $G(x,y)\mapsto Gy$, and the product $(G(x,y),G(y,z))\mapsto G(x,z)$.} and let $H\subset G$ be the set of all non-identity morphisms. Then the triple $(G,H,j)$, where
$j\colon B(aB,a'B)\mapsto B(a^{-1}B,a^{-1}a'B)$, is a $\Delta$-groupoid.
\end{example}
\begin{remark}
The constructions in Examples~\ref{myex:1} and \ref{myex:2} are functorial \cite{K3} and there exist the respective left adjoint functors $A',B'\colon \myK\to\myR$ defined as follows.
Let $G=(G,H,j)$ be a $\Delta$-groupoid. The ring $A'G$ is the quotient ring of the groupoid ring $\mathbb{Z}[G]$  (generated over $\mathbb{Z}$ by the elements $\{ w_x\vert\ x\in G\}$ with the defining relations $w_xw_y=w_{xy}$ if $x,y$ are composable) with respect to the
additional relations  $w_x+w_{k(x)}=1$ for all $x\in H$. The ring $B'G$ is generated over $\mathbb{Z}$ by the elements $\{ u_x, v_x\vert\ x\in G\}$ with the defining relations $u_xu_y=u_{xy}, v_{xy}=u_xv_y+v_x$ if $x,y$ are composable, and $u_{k(x)}=v_x, v_{k(x)}=u_x$ for all $x\in H$.
\end{remark}
\begin{remark}\label{R:nat-tr}
There exists a functorial morphism $\alpha\colon B'\to A'$ defined by the relations
\[
\alpha_G(u_x)=w_x,\ \alpha_G(v_x)=1-w_x.
\]
Moreover,  $\alpha_G$ induces a ring isomorphism $A'G\simeq B'G/I_G$, where $I_G$ is the two-sided ideal generated by the set $\{u_x+v_x-1\vert\ x\in H\}$.
\end{remark}

\subsection{Labeled oriented truncated tetrahedra}
In what follows, in a $\Delta$-groupoid $(G,H,j)$, we will use the following binary operation on the set of $H$-composable pairs: $x*y=j(k(x)j(y))$.

Consider a tetrahedron with totally ordered vertices. The vertex order induces natural orientation of long edges of the corresponding truncated tetrahedron, where the orienting arrows point from smaller to bigger vertices. In each hexagonal face we choose the unique oriented short edge according to this picture
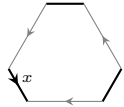
\begin{center}
\begin{tikzpicture}[>=stealth,scale=0.5]
\draw[thick](0,0)--(120:1);\draw[thick](2,0)--+(60:1);\draw[thick](60:3)--+(-1,0);
\draw[black!50](0,0)--(2,0);\draw[black!50](120:1)--+(60:2);\draw[black!50](60:3)--+(-60:2);
\draw[thick,->] (120:.501)--(120:.5);
\draw[black!50,->](1.01,0)--(1,0);\draw[black!50,->](120:1)+(60:1)--+(60:.99);
\draw[black!50,->](60:3)+(-60:1.01)--+(-60:1);
\draw (0,.6) node{\tiny $x$};
\end{tikzpicture}
\end{center}
where the short edge $x$ is characterized by the following properties:
 \begin{enumerate}
 \item it is located near the biggest (truncated) vertex of the face;
 \item its orientation is antiparallel to the orientation of the opposite long edge.
  \end{enumerate}
  With this choice, our \emph{labeled oriented truncated (l.o.t.)} tetrahedron looks as in this picture
  
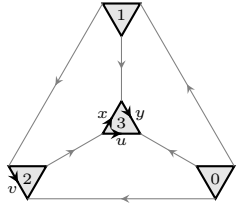
\begin{center}
\begin{tikzpicture}[>=stealth,scale=.5]
\draw[thick,fill=black!10] (0,0)--(60:1)--(1,0)--cycle;\draw[thick,fill=black!10] (-1,0)++(-120:1)--++(-1,0)--+(-60:1)--cycle;
\draw[thick,fill=black!10] (60:2)++(120:1)--++(60:1)--+(-1,0)--cycle;\draw[thick,fill=black!10] (2,0)++(-60:1)--++(1,0)--+(-120:1)--cycle;
\draw[black!50] (-1,0)+(-120:1)--(0,0);\draw[black!50,->] (-150:.87)--(-150:.86);
\draw[black!50] (2,0)+(-60:1)--(1,0);\draw[black!50,->] (1,0)+(-30:.87)--+(-30:.86);
\draw[black!50] (60:2)+(120:1)--(60:1);\draw[black!50,->] (60:1)+(0,.87)--+(0,.86);
\draw[black!50] (-2,0)++(-120:1)--+(60:5);\draw[black!50,->] (-2,0)++(-120:1)+(60:2.5)--+(60:2.49);
\draw[black!50] (3,0)++(-60:1)--+(120:5);\draw[black!50,->] (3,0)++(-60:1)+(120:2.5)--+(120:2.51);
\draw[black!50] (-120:2)++(-1,0)--+(5,0);\draw[black!50,->] (-120:2)++(-1,0)+(2.5,0)--+(2.49,0);
\draw[thick,->] (60:.5)--(60:.501);\draw[thick,->] (60:1)+(-60:.5)--+(-60:.501);
\draw[thick,->] (.5,0)--(.501,0);
\draw[thick,->] (-2,0)++(-120:1)+(-60:.5)--+(-60:.501);
\draw (0,.5) node{\tiny $x$};\draw (1,.5) node{\tiny $y$};
\draw (.5,-.2) node{\tiny $u$};\draw (-2.4,-1.5) node{\tiny $v$};
\draw (3,-1.2) node{\tiny $0$};\draw (-2,-1.2) node{\tiny $2$};
\draw (.5,3.15) node{\tiny $1$};\draw (.5,.3) node{\tiny $3$};
\end{tikzpicture}
 \end{center}
 where $u=xy$, $v=x*y$, and the numbers within triangular faces enumerate the vertices (according to their order) of the non-truncated tetrahedron. As each hexagonal face is opposite to a unique vertex of the non-truncated tetrahedron, we have an induced order in the set of hexagonal faces.
To express all this information, it will be convenient to use the following graphical notation for such l.o.t. tetrahedra:
 \begin{center}
 \begin{tikzpicture}[scale=.3]
 \draw[very thick] (0,0)--(3,0);
 \draw (0,0)--(0,1);\draw (1,0)--(1,1);\draw (2,0)--(2,1);\draw (3,0)--(3,1);
 \draw (0,1.3) node{\tiny $x$};\draw (1,1.3) node{\tiny $u$};\draw (2,1.3) node{\tiny $y$};\draw (3,1.3) node{\tiny $v$};
 \end{tikzpicture}
 \end{center}
 where the vertical labeled segments, ordered from left to right, correspond to the hexagonal faces with their induced order, and the labels correspond to the distinguished oriented short edges.
\section{Examples of $\Delta$-groupoids from ideal triangulations}\label{mysec:2}
In this section, by considering concrete examples, we show how can one associate a $\Delta$-groupoid to ideal triangulations. One can say that an ideal triangulation gives a presentation of a $\Delta$-groupoid in the same sense as a cell complex gives a presentation for its edge-path groupoid. We also calculate the associated rings obtained by applying the adjoint functors $A'$ and $B'$.
\subsection{The trefoil knot}
There is an ideal triangulation of the complement of the trefoil knot consisting of two tetrahedra described by the following diagram
\begin{equation}\label{P:trefoil}
\begin{tikzpicture}[baseline=5pt,scale=.7]
\draw[very thick] (0,0)--(3,0);\draw[very thick] (0,1)--(3,1);
\draw[thin] (0,0)--node[left]{\tiny $x$}(0,1);
\draw (1,0)--node[left]{\tiny $u$}(1,1);
\draw (2,0)--node[left]{\tiny $y$}(2,1);
\draw (3,0)--node[left]{\tiny $v$}(3,1);
\end{tikzpicture}
\end{equation}
Using our conventions, we obtain a presentation of the associated $\Delta$-groupoid $G$ given by four generators $x,y,u,v$ and four relations:
\begin{subequations}\label{E:defrel}
\begin{align}\label{E:defrel1}
u&=xy,\\ \label{E:defrel2}v&=x*y,\\ \label{E:defrel3}y&=vu,\\ \label{E:defrel4}x&=v*u.
\end{align}
\end{subequations}
Excluding generators $u$ and $v$ by using equations~\eqref{E:defrel1} and \eqref{E:defrel2}, we obtain an equivalent presentation with two generators $x,y$ and two relations:
\begin{subequations}\label{E:defrelred}
\begin{align}\label{E:defrelred1}
y&=(x*y)xy,\\ \label{E:defrelred2}x&=(x*y)*(xy).
\end{align}
\end{subequations}
Equation~\eqref{E:defrelred1} implies that
\[
i(x)=x*y=j(k(x)j(y))\Leftrightarrow ji(x)=k(x)j(y)\Leftrightarrow y=j((ji(x))^2).
\]
Combining this with \eqref{E:defrelred2} we obtain
\[
x=j(k(x*y)j(xy))=j(ki(x)j(x)i(x))=ji(x).
\]
Thus, all the relations~\eqref{E:defrel} of the presentation of our $\Delta$-groupoid reduce to the relations
\[
i(x)=j(x)=k(x)=v,\quad y=j(x^2),\quad u=j(x^{-2}).
\]
\subsubsection{The $A'$-ring}
\begin{theorem}
The $A'$-ring associated to ideal triangulation~\eqref{P:trefoil} is isomorphic to the ring
\(
\mathbb{Z}[t,3^{-1}]/(\Delta_{3_1}(t)),
\)
where $\Delta_{3_1}(t)=t^2-t+1$ is the Alexander polynomial of the trefoil knot.
\end{theorem}
\begin{proof}
We have the following presentation
\[
A'G=\mathbb{Z}\langle w_x^{\pm1},w_y^{\pm1}\vert\ w_x^{-1}=1-w_x,w_y^{-1}=1-w_x^{-2}\rangle.
\]
The maps
\[
\phi(w_x)=t,\ \phi(w_y)=3^{-1}(2-t),
\]
specify a well defined ring homomorphism $\phi\colon A'G\to\mathbb{Z}[t,3^{-1}]/(\Delta_{3_1}(t))$. Indeed, we have
\[
\phi(w_x^{-1})=\phi(w_x)^{-1}=t^{-1}=1-t=1-\phi(w_x)=\phi(1-w_x),
\]
 and
 \begin{multline*}
 \phi(w_y^{-1})=\phi(w_y)^{-1}=3(2-t)^{-1}=3(1+t^{-1})^{-1}\\=3((1+t^{-1})(1+t))^{-1}(1+t)=
 1+t=1-t^{-2}=1-\phi(w_x)^{-2}=\phi(1-w_x^{-2}).
 \end{multline*}
 On the other hand, it is easily verified that the inverse of $\phi$ has the form:
 \[
 \phi^{-1}(t)=w_x,\ \phi^{-1}(3^{-1})=w_xw_y^2.
 \]
\end{proof}
\subsubsection{The $B'$-ring}
\begin{theorem}
The $B'$-ring associated to ideal triangulation~\eqref{P:trefoil} is isomorphic to its $A'$-ring.
\end{theorem}
\begin{proof}
Due to Remark~\ref{R:nat-tr}, it is enough to show that $u_x+v_x=1$, $x\in H$, in the presentation of the $B'$-ring. Indeed, we have the presentation
\begin{multline*}
B'G=\mathbb{Z}\langle u_x^{\pm1}, v_x^{\pm1},u_y^{\pm1}, v_y^{\pm1}\vert\\ u_x^{-1}=v_x,\ -u_x^{-1}v_x=u_x,\ u_x^2=-v_y^{-1}u_y,\ (u_x+1)v_x=v_y^{-1}\rangle
\end{multline*}
Denoting $u_x=t$, from the first, forth  and third relations we have
\[
v_x=t^{-1},\ v_y=(1+t^{-1})^{-1},\ u_y=-(1+t^{-1})^{-1}t^2,
\]
while the second relation implies that
\[
t^3+1=(t+1)(t^2-t+1)=0.
\]
Taking into account invertibility of $v_y$, we conclude that $t^2-t+1=0$. Thus,
\[
u_x+v_x=t+t^{-1}=1,\quad u_y+v_y=(1+t^{-1})^{-1}(t^{-1}+1)=1.
\]
\end{proof}
\subsection{The figure-eight knot}
The standard ideal triangulation of the complement of the figure-eight knot \cite{Thurston} also consists of two tetrahedra and, in our notation, is described by the following diagram
\begin{equation}\label{E:8graph}
\begin{tikzpicture}[baseline=5pt,scale=.7]
\draw[very thick] (0,0)--(3,0);\draw[very thick] (0,1)--(3,1);
\draw[thin] (0,0)--node[right,near end]{\tiny $x$}(1,1);
\draw (1,0)--node[right,near start]{\tiny $u$}(0,1);
\draw (2,0)--node[left,near start]{\tiny $y$}(3,1);
\draw (3,0)--node[left,near end]{\tiny $v$}(2,1);
\end{tikzpicture}
\end{equation}
The corresponding presentation of the associated $\Delta$-groupoid is given by four generators $x,y,u,v$ and four relations:
\begin{subequations}\label{E:8defrel}
\begin{align}\label{E:8defrel1}
u&=xy,\\ \label{E:8defrel2}v&=x*y,\\ \label{E:8defrel3}v&=yx,\\ \label{E:8defrel4}u&=y*x.
\end{align}
\end{subequations}
which, by excluding the variables $u$ and $v$, reduces to a presentation with two generators $x$ and $y$ and two relations:
\begin{subequations}\label{E:8defrelred}
\begin{align}\label{E:8defrelred1}
xy&=y*x\Leftrightarrow j(xy)=k(y)j(x),\\ \label{E:xydefrelred2}yx&=x*y\Leftrightarrow j(yx)=k(x)j(y).
\end{align}
\end{subequations}
To analyze the associated rings of this $\Delta$-groupoid, we shall find it useful to use a class of rings which contains the rings of two-by-two matrices over commutative rings.
\subsubsection{Formal $M_2$-rings}
A ring $R$ with center $Z(R)$ is called \emph{formal $M_2$-ring}  if there exists a homomorphism of Abelian groups $L\colon R\to Z(R)$, called \emph{trace function}, such that
\[
L(1)=2,\quad
L(x)x-x^2\in Z(R),\quad \forall x\in R.
\]
 Thus, a formal $M_2$-ring is a ring where any element satisfies a quadratic equation over its center. For a given formal $M_2$-ring the trace function $L$ is not necessarily unique. Indeed, for any Abelian group homomorphism $M\colon R\to Z(R)$ satisfying the conditions $M(x)x\in Z(R)$ for all $x\in R$, and $M(1)=0$, the function $L+M$ is also a trace function.

  In a formal $M_2$-ring $R$, the map $Q\colon R\to Z(R)$, defined by the formula $Q(x)=L(x)x-x^2$ is a quadratic form on the $\mathbb{Z}$-module $R$ such that $Q(1)=1$, and one has the identity
\begin{equation}\label{E:anti-com}
xy+yx=-(x,y)+L(x)y+L(y)x,\quad \forall x,y\in R,
\end{equation}
where
\begin{equation}\label{E:bilin}
(x,y)=Q(x+y)-Q(x)-Q(y)
\end{equation}
is (twice) the symmetric $\mathbb{Z}$-bilinear form associated with $Q$. In particular, specifying $y=1$ in \eqref{E:anti-com}, we obtain the equality
\[
L(x)=(x,1),\quad \forall x\in R.
\]

A formal $M_2$-ring $R$ with a trace function $L$ is called \emph{symmetric} if $L\circ L=2L$ and the following identity is satisfied:
\[
(x,y)=L(x)L(y)-L(xy),\quad \forall x,y\in R.
\]
 In a symmetric formal $M_2$-ring $R$ one has the additional properties:
 \[
L\circ Q=2Q,\quad  L(xy)=L(yx),\quad Q(xy)=Q(x)Q(y),\quad \forall x,y\in R.
 \]
 A basic example of a symmetric formal $M_2$-ring is the ring $M_2(K)$ of two-by-two matrices over a commutative ring $K$, where the trace function is given by the usual trace: $L(x)=\mathrm{tr}(x)$.

 \begin{lemma}\label{L:cent}
 Let a ring $R$ be generated over its center $Z(R)$ by two elements $a$ and $b$. Then, the condition
 \[
 ab+ba\in Z(R)+Z(R)a+Z(R)b
 \]
 is satisfied if and only if the conditions
 \[
 a^2\in Z(R)+Z(R)a,\quad b^2\in Z(R)+Z(R)b
 \]
 are satisfied.
 \end{lemma}
 \begin{proof}
 For some $x,y\in Z(R)$, consider the elements
 \[
 z=ya+xb-ab-ba,\ p=xa-a^2,\ q=yb-b^2.
 \]
 Denoting $[u,v]_\pm=uv\pm vu$ for any $u,v\in R$, we have
\begin{multline*}
[a,z]_-=x[a,b]_--[a,[a,b]_+]_-=x[a,b]_--[a,[a,b]_-]_+\\=x[a,b]_--[a^2,b]_-
=[xa-a^2,b]_-=[p,b]_-
\end{multline*}
and, by a similar calculation, $[b,z]_-=[q,a]_-$. Thus, taking into account the evident equalities $[a,p]_-=[b,q]_-=0$, we conclude that $z$ is central if and only if $p$ and $q$ are central.
 \end{proof}
 \begin{lemma} \label{L:for-mat}
Let $x,y,p,q$ be some elements in a commutative ring $K$. Then the ring $R$, defined by the presentation
\begin{equation}\label{E:comp-ring}
R(K,x,y,p,q)=K\langle a,b\vert\ a^2=xa-p,\ b^2=yb-q\rangle,
\end{equation}
is a symmetric formal $M_2$-ring with the trace function defined by the formulae
\begin{equation}\label{E:tr-fun}
L(a)=x,\quad L(b)=y.
\end{equation}
\end{lemma}
\begin{proof}
By Lemma~\ref{L:cent}, the element
\[
z=ya+xb-ab-ba
\]
 is central.
Any element of $R(K,x,y,p,q)$ is a linear combination of the elements $1,a,b,ab$ with coefficients in the polynomial ring $K[z]$. Indeed, by using the relation $ba=-ab+ya+xb-z$, any word in letters $a$ and $b$ can be written as a finite sum of the form $\sum_{m,n\in\mathbb{Z}_{\ge0}}r_{m,n}a^mb^n$ with $r_{m,n}\in K[z]$. Due to the relation $a^2=xa-p$, any power $a^m$ is a linear combination of $1$ and $a$ with coefficients in $K$. Similarly, due to the relation $b^2=yb-q$, any power $b^n$ is a linear combination of $1$ and $b$ with coefficients in $K$. Thus, any word in $a$ and $b$ is reduced to a unique linear combination of the words $1,a,b,ab$ with coefficients in $K[z]$. Moreover, by considering a representation into the matrix ring $M_4(K[z])$ of the form
\[
a\mapsto\begin{pmatrix}
0&-p&0&0\\
1&x&0&0\\
0&0&0&-p\\
0&0&1&x
\end{pmatrix},\quad
b\mapsto\begin{pmatrix}
0&-z&-q&-qx\\
0&y&0&q\\
1&x&y&xy-z\\
0&-1&0&0
\end{pmatrix},
\]
it can be verified that the reduced form of any element of $R(K,x,y,p,q)$ is unique.

It can be now straightforwardly verified that $R$ is a symmetric formal $M_2$-ring with the trace function $L$ defined by the formulae~\eqref{E:tr-fun} and extended linearly over $K[z]$ to the whole ring $R(K,x,y,p,q)$. Indeed, we have
\[
L(x)=L(L(a))=2L(a)=2x,\quad L(p)=L(Q(a))=2Q(a)=2p,
\]
and similarly for $y$ and $q$. Then, as
\[
(a,b)=(x+y)(a+b)-(a+b)^2-p-q=xb+ya-ab-ba=z,
\]
we also have $L(z)=2z$. Thus, for any $u\in\mathbb{Z}[x,y,z,p,q]$, we have $L(u)=2u$.

\end{proof}
\subsubsection{Ideals in formal $M_2$-rings}
Let $I$ be a two-sided ideal in a  formal $M_2$-ring $R$ with a trace function $L\colon R\to Z(R)$. Then, evidently,
\(
Q(I)\subset I
\)
but not necessarily $L(I)\subset I$. However, from the identity
\[
xL(y)-(x,y)\in I,\quad x\in R,\ y\in I,
\]
it follows that $RL(I)\subset I+Z(R)$. Moreover, in a symmetric formal $M_2$-ring the relation
\begin{equation}\label{E:id-incl}
xL(y)-L(x)L(y)+L(xy)=xL(y)+L(xy-L(x)y)\in I,\quad x\in R,\ y\in I,
\end{equation}
implies that $RL(I)\subset L(I)+I$.
Thus, the set $L(I)+I$ is a two-sided ideal in $R$ and its image in the quotient ring $R/I$ is an ideal contained in the center $Z(R/I)$. It is clear that the quotient ring $R/(L(I)+I)$ inherits the structure of a formal $M_2$-ring with the trace function induced from that of $R$.
\begin{remark}\label{R:ideal-elements-by-traces}
Equation~\eqref{E:id-incl} implies that if $y\in I$ is such that $L(y)\in I$, then $L(xy)\in I$ for any $x\in R$. This fact can be used for construction of elements of the intersection $I\cap Z(R)$.
\end{remark}
\subsubsection{The $A'$-ring}
The ring $A'G$ has the initial presentation
\begin{multline*}
\mathbb{Z}\langle w_x^{\pm1},w_y^{\pm1},(1-w_x)^{-1},(1-w_y)^{-1},(1-w_xw_y)^{-1}\vert\\ (1-(w_xw_y)^{-1})^{-1}=(1-w_y)(1-w_x^{-1})^{-1},\\ (1-(w_yw_x)^{-1})^{-1}=(1-w_x)(1-w_y^{-1})^{-1}\rangle.
\end{multline*}
which can be shown to be equivalent to the quotient ring $R(\mathbb{Z}[s],s,s,1,1)/I$, where the two-sided ideal $I$ is generated by the elements
\[
(s-1)(a+1)-z,\quad (s-1)(b+1)-z.
 \]
 Identification with the initial presentation is given by the formulae:
 \[
 w_x\mapsto a,\quad w_y\mapsto b.
 \]
The techniques of formal $M_2$-rings lead to an equivalent presentation of the form
\begin{multline}
\mathbb{Z}[s,z]/((s-1)(z-2),s^2-2z-1,z^2-2z)\langle a,g\vert\ \  a^2-a=z-s,\ g^2=z-2,\\
 ag+ga=2-z+g,\
sa=a-s+z+1,\ sg=g,\ za=s-1,\ zg=0\rangle,
\end{multline}
with the identifications
\[
w_x\mapsto a,\quad w_y\mapsto a+g.
\]
The part $J=\mathbb{Z}(s-1)+\mathbb{Z}z$ in this presentation coincides with the image of the ideal $L(I)+I$. The quotient ring $A'G/J$, having a unique structure of a symmetric formal $M_2$-ring, is isomorphic to the ring of Hurwitz quaternions.

\subsubsection{The $B'$-ring}

\begin{theorem}\label{P:b-prime-f8}
The $B'$-ring associated to diagram~\eqref{E:8graph} admits the following presentation
\[
\mathbb{Z}\left\langle a,b,c^{\pm1}\left\vert\ c=a(a+1),\ c^{-1}=b(b+1),\ c=\left(aba^{-1}b^{-1}\right)^2\right.\right\rangle.
\]
\end{theorem}
\begin{proof}
Relations~\eqref{E:8defrelred} correspond to the presentation
\begin{multline*}
B'G=\mathbb{Z}\langle u_x^{\pm1},v_x^{\pm1},u_y^{\pm1},v_y^{\pm1}\vert \\ (u_xv_y+v_x)^{-1}u_xu_y=v_yv_x^{-1}u_x,\ (u_xv_y+v_x)^{-1}=(v_yv_x^{-1}+u_y),\\ (u_yv_x+v_y)^{-1}u_yu_x=v_xv_y^{-1}u_y,\ (u_yv_x+v_y)^{-1}=(v_xv_y^{-1}+u_x)\rangle
\end{multline*}
Introducing two new elements
\[
a=v_x^{-1}u_xv_y,\quad b=v_y^{-1}u_yv_x,
\]
we obtain an equivalent presentation
\begin{multline}\label{E:pres0}
B'G\simeq\mathbb{Z}\langle u_x^{\pm1},v_x^{\pm1},u_y^{\pm1},v_y^{\pm1},a^{\pm1},b^{\pm1}\vert \
u_x=v_xav_y^{-1},\ u_y=v_ybv_x^{-1},\\
(a+1)^{-1}abv_x^{-1}=v_yav_y^{-1},\ (a+1)^{-1}=v_y(1+b),\\
(b+1)^{-1}bav_y^{-1}=v_xbv_x^{-1},\ (b+1)^{-1}=v_x(1+a)\rangle.
\end{multline}
The forth and sixth relations in \eqref{E:pres0} can be used to eliminate the generators $v_x$ and $v_y$, thus obtaining another presentation
\begin{multline}\label{E:pres}
B'G\simeq\mathbb{Z}\langle u_x^{\pm1},v_x^{\pm1},u_y^{\pm1},v_y^{\pm1},a^{\pm1},b^{\pm1}\vert \
u_x=v_xav_y^{-1},\ u_y=v_ybv_x^{-1},\\
v_y^{-1}=(b+1)(a+1),\ v_x^{-1}=(a+1)(b+1),\\
ab(a+1)(b+1)=(b+1)^{-1}a(b+1)(a+1),\ ba(b+1)(a+1)=(a+1)^{-1}b(a+1)(b+1)
\rangle.
\end{multline}
The last two  relations in this presentation imply the equality
\[
b(b+1)a(a+1)=1
\]
which means that the invertible element $c=a(a+1)$ is central and  $c^{-1}=b(b+1)$. Solving these equations with respect to $a+1$ and $b+1$ and substituting them into any of the last two relations of \eqref{E:pres}, we obtain
\[
ba(ab)^{-1}=c^{-1}ab(ba)^{-1}\Leftrightarrow c=(ab(ba)^{-1})^2.
\]
\end{proof}
The ring
\[
R=\mathbb{Z}\left\langle a,b,c^{\pm1}\left\vert\ c=a(a+1),\ c^{-1}=b(b+1)\right.\right\rangle
\]
is isomorphic to the ring $R(\mathbb{Z}[c,c^{-1}],-1,-1,-c,-c^{-1})$ and, due to Lemma~\ref{L:for-mat}, is a symmetric formal $M_2$-ring. In fact, it is a four-dimensional algebra over the polynomial ring
$\mathbb{Z}[d,c,c^{-1}]$, where the element $d=ab+ba +a+b$ is central, and the elements $1,a,b,ab$ constitute a linear basis. The trace function is given by a $\mathbb{Z}[d,c,c^{-1}]$-linear map defined by the relations
\[
L(1)=2,\ L(a)=L(b)=-1,\ L(ab)=1+d.
\]
 In particular, we have the following identity:
\begin{equation}\label{E:ss}
q+q^{-1}=w\in\mathbb{Z}[d,c,c^{-1}],\quad q=aba^{-1}b^{-1},\quad w=d^2+d-c-c^{-1}-2.
\end{equation}
Indeed, we have
 \[
 q=ab(a+1)(b+1)=(d+1)ab+(d+c^{-1}+1)a-cb-c-1,
 \]
 and, by the symmetry $a\leftrightarrow b$, $c\leftrightarrow c^{-1}$, we also have
 \[
 q^{-1}=ba(b+1)(a+1)=(d+1)ba+(d+c+1)b-c^{-1}a-c^{-1}-1.
 \]
 Adding these equalities, we obtain
 \[
 q+q^{-1}=(d+1)(ab+ba+a+b)-c-c^{-1}-2
 \]
 which is equivalent to equality~\eqref{E:ss}.

Thus, by Theorem~\ref{P:b-prime-f8}, the ring $B'G$ is isomorphic to the quotient  ring $R/I$, where $I$ is the two sided ideal generated by the element $\xi=q^2-c=wq-1-c$,
Analysis of the structure of the ring $B'G$ gives rise to the following description.

\begin{lemma}\label{L:ideal-elements}
 The following elements of $R$ are in the ideal $I$:
\begin{multline}\label{E:some-rel}
5\lambda,\ \lambda^2,\ (a-2)\lambda,\ (b-2)\lambda, \ (w-2)\lambda,\ (d-2)\lambda,\ (c-1)\lambda,\ (q-1)\lambda,\\
\lambda+2(c-c^{-1}),\ w^2-2w,\ d^2+d-3w,
\end{multline}
where $\lambda$ is either of the elements $c^{\pm1}+1-w\in R$.
\end{lemma}
\begin{proof}
In the ring $R$ we have the following identities
\[
(a^{-1}\xi-\xi a^{-1})b^{-1}+(d-a)\xi=(a-d)(c+1-w),
\]
\[
(b^{-1}\xi b-\xi)a^{-1}b^{-1}+(d-b)\xi=(b-d)(c+1-w).
\]
That means that the elements $(a-d)(c+1-w)$ and $(b-d)(c+1-w)$ belong to the ideal $I$. As this ideal is invariant with respect to the symmetry $a\leftrightarrow b$, $c\leftrightarrow c^{-1}$, the elements $(a-d)(c^{-1}+1-w)$ and $(b-d)(c^{-1}+1-w)$ also belong to $I$.

Let $\lambda\in \{c^{\pm1}+1-w\}$. We have the inclusions
\[
\{(a-d)\lambda, (b-d)\lambda\}\subset I.
\]
Using the definition of the element $d$ in terms of $a$ and $b$ we have
\begin{equation}\label{E:d-theta}
(d-2d(d+1))\lambda\in I.
\end{equation}
On the other hand, from the definition  of $q$ it follows that
\[
(q-1)\lambda=((a-d)b+b(d-a))(ba)^{-1}\lambda\in I,
\]
which implies the following sequence of inclusions
\[
(c-1)\lambda=-\xi\lambda+(q^2-1)\lambda\in I,
\]
\[
(d(d+1)-1)\lambda=(d(d+1)-a(a+1))\lambda+(c-1)\lambda\in I,
\]
\[
(d-2)\lambda=(d-2d(d+1))\lambda+2(d(d+1)-1)\lambda\in I,
\]
\[
(w-2)\lambda=((d(d+1)-6)+(2-c-c^{-1}))\lambda\in I,
\]
\[
\lambda^2=((c^{\pm1}-1)+(2-w))\lambda\in I,
\]
and
\[
5\lambda=(6-d(d+1))\lambda+(d(d+1)-1)\lambda\in I.
\]
The  identity
\[
w-1-c-2(c-c^{-1})=(5+3(q-1))(1+c^{-1}-w)+3\xi(1-q^{-1}c^{-1})
\]
is satisfied due to the relation $q+q^{-1}=w$, and its right hand side is manifestly an element of $I$. By a similar identity, obtained by the replacements $q\leftrightarrow q^{-1}$, $c\leftrightarrow c^{-1}$, we have $w-1-c^{-1}+2(c-c^{-1})\in I$. We have two more identities
\[
d(d+1)-3w=(c+1-w+2(c-c^{-1}))+(c^{-1}+1-w-2(c-c^{-1}))
\]
and
\[
w^2-2w=\xi(1-q^{-2}c^{-1})+(d(d+1)-3w)
\]
where the right hand sides are also manifestly in $I$.
\end{proof}
\begin{lemma}
The ideal $L(I)+I$ is generated by the elements $\epsilon=c+1-w$ and $\xi=q^2-c$.
\end{lemma}
\begin{proof}
As elements $1,a,b,ba$ constitute a linear basis in $R$ over its center, and due to the cyclic property of the trace function, the trace of  any element in the ideal $I$ is a linear combination of the elements
\[
L(\xi),\ L(\xi a),\ L(\xi b),\ L(\xi ba)
\]
with coefficients in $Z(R)$. Calculating these four traces, we obtain
\[
L(\xi)=-2\epsilon+w^2-2w,\ L(\xi a)=\epsilon+2w-w^2,\ L(\xi b)=\epsilon,\ L(\xi ba)=(1+d)\epsilon.
\]
\end{proof}
As below we will work with elements of the quotient ring $R/I$, in order to  simplify the notation, for all $x\in R$, we will denote by the same letter $x$ the corresponding class $x+I\in R/I$.
\begin{theorem}\label{T:main}
Any element $x\in R/I$ can uniquely be  written as the following linear combination
\[
x=m\epsilon+\sum_{\mu\in\{1,w,d\}}\sum_{\nu\in\{1,a,b,ab\}}n_{\mu,\nu}\,\mu\nu,
\]
where $m\in \{0,\pm1,\pm2\}$ and $n_{\mu,\nu}\in\mathbb{Z}$. The ring structure of $R/I$ is characterized by the conditions that the elements $\epsilon,w,d$ are central, and the multiplication rules:
\[
\epsilon^2=0,\ 5\epsilon=0,\ \epsilon \mu\nu= \varepsilon_\mu\varepsilon_\nu\,\epsilon,\quad \varepsilon_x=\left\{\begin{array}{cl}
1,& \mathrm{if}\ x=1;\\
2,& \mathrm{if}\ x\in\{w,d,a,b\};\\
4,& \mathrm{if}\ x=ab;
\end{array}\right.
\]
\[
a^2=\epsilon-1+w-a,\ b^2=-\epsilon-1+w-b,\ ba=d-a-b-ab;
\]
\[
w^2\nu=2w\nu,\ d^2\nu=-d\nu+3w\nu,\quad \nu\in\{1,a,b,ab\};
\]
\[
wd\nu=\left\{\begin{array}{cl}
\epsilon+wa+wb+2wab,& \mathrm{if}\ \nu=1;\\
2\epsilon+w-wa+2wb-wab,& \mathrm{if}\ \nu=a;\\
2\epsilon+w+2wa-wb-wab,& \mathrm{if}\ \nu=b;\\
-\epsilon+2w-wa-wb,& \mathrm{if}\ \nu=ab;\\
\end{array}\right.
\]
\[
a^2b=2\epsilon -b+wb-ab,\ ab^2=-2\epsilon-a+wa-ab;
\]
\[
aba=2\epsilon+1-w+a+da +b-wb;
\]
\[
bab=-2\epsilon+1-w+a-wa+b+db;
\]
\[
(ab)^2=-1+ab+dab.
\]
\end{theorem}
\begin{proof}
Let $K\subset Z(R/I)$ be the sub-ring generated by the elements $w$ and $d$. The relations $w^2=2w$, $d^2=-d+3w$ and $wd=dw$ imply that $K$ is a four-dimensional commutative $\mathbb{Z}$-algebra, with elements $1$, $w$, $d$, and $wd$ as a linear basis. Any element $x\in R/I$ can be written in the form
\begin{equation}\label{E:generic-element}
x=m\epsilon+x_0+x_1a+x_2b+x_3ab,\quad m\in \{0,\pm1,\pm2\},\quad x_0,x_1,x_2,x_3\in K,
\end{equation}
but such form is not unique. Indeed,
by using Lemma~\ref{L:ideal-elements}, we have the identity:
\[
w(-d+a+b+2ab)=w(ab-ba)=w(q-1)ba=(c+1-w)ba=\epsilon ba=4\epsilon=-\epsilon,
\]
which can be solved for the product $wd$:
\[
wd=\epsilon+wa+wb+2wab.
\]
By repeated use of this relation, we can eliminate all the products $wd$ in formula~\eqref{E:generic-element}. The resulting formula for $x$, called \emph{reduced form},  takes the same form as before but each coefficient $x_i$, $i\in\{0,1,2,3\}$, is a $\mathbb{Z}$-linear combination of $1$, $w$ and $d$. To prove that the reduced form is unique it is enough to check that $x=0$ if and only if its reduced form vanishes. This can be straightforwardly checked by using the representation in the endomorphism ring of the Abelian group $(\mathbb{Z}/5\mathbb{Z})\oplus\mathbb{Z}^{12}$  obtained by left multiplications. Namely, with respect to the ordered basis
\(
(\epsilon,1,w,d,a,wa,da,b,wb,db,ab,wab,dab)
\)
this representation takes the form
\[
a\mapsto \left(\begin{array}{ccccccccccccc}
2& 0& 0& 0& 1& 2& -2& 0& 0& 0& 2& -1& 1\\
0& 0& 0& 0& -1& 0& 0& 0& 0& 0& 0& 0& 0\\
0& 0& 0& 0& 1& 1& 0& 0& 0& 0& 0& 0& 1\\
0& 0& 0& 0& 0& 0& -1& 0& 0& 0& 0& 0& 0\\
0& 1& 0& 0& -1& 0& 0& 0& 0& 0& 0& 0& 0\\
0& 0& 1& 0& 0& -1& 1& 0& 0& 0& 0& 0& 2\\
0& 0& 0& 1& 0& 0& -1& 0& 0& 0& 0& 0& 0\\
0& 0& 0& 0& 0& 0& 0& 0& 0& 0& -1& 0& 0\\
0& 0& 0& 0& 0& 0& 1& 0& 0& 0& 1& 1& -1\\
0& 0& 0& 0& 0& 0& 0& 0& 0& 0& 0& 0& -1\\
0& 0& 0& 0& 0& 0& 0& 1& 0& 0& -1& 0& 0\\
0& 0& 0& 0& 0& 0& 2& 0& 1& 0& 0& -1& -1\\
0& 0& 0& 0& 0& 0& 0& 0& 0& 1& 0& 0& -1
\end{array}
\right),
\]
\[
b\mapsto
\left(
\begin{array}{ccccccccccccc}
2& 0& 0& 0& 0& 1& 0& -1& -2& -1& -2& -2& -2\\
0& 0& 0& 0& 0& 0& 0& -1& 0& 0& 1& 0& 0\\
0& 0& 0& 0& 0& 0& 3& 1& 1& 0& -1& 0& -1\\
0& 0& 0& 0& 1& 0& -1& 0& 0& -1& 0& 0& 1\\
0& 0& 0& 0& -1& 0& 0& 0& 0& 0& 1& 0& 0\\
0& 0& 0& 0& 0& 0& 0& 0& 0& 1& -1& 1& 0\\
0& 0& 0& 0& 0& 0& -1& 0& 0& 0& 0& 0& 1\\
0& 1& 0& 0& -1& 0& 0& -1& 0& 0& 1& 0& 0\\
0& 0& 1& 0& 0& 0& 0& 0& -1& 1& 0& 0& 0\\
0& 0& 0& 1& 0& 0& -1& 0& 0& -1& 1& 0& 0\\
0& 0& 0& 0& -1& 0& 0& 0& 0& 0& 0& 0& 0\\
0& 0& 0& 0& 0& 1& 0& 0& 0& 2& 0& -1& -1\\
0& 0& 0& 0& 0& 0& -1& 0& 0& 0& 0& 0& 0
\end{array}
\right),
\]
where the numbers on the first lines of these matrices are considered modulo $5$.
\end{proof}
\subsubsection{Remarks}
\begin{enumerate}

\item A $\mathbb{Z}$-linear basis of the center $Z(R/I)$ is given by the elements
\[
\epsilon,\ 1,\ w,\ d,\ p=5wa,\ q=5wb,\ r=wab-2wa-2wb.
\]
\item The quotient ring $R/(L(I)+I)$, which is identified with $(R/I)/\mathbb{Z}\epsilon$, is a symmetric formal $M_2$-ring with the trace function defined by the formulae
    \[
   L(1)=2,\ L(a)=L(b)=-1,\ L(ab)=1+d,
    \]
   and extended linearly over the polynomial ring  $\mathbb{Z}[w,d]$. The elements
   \[
   1,\ w,\ d,\ e=wa+wb+2wab,\ f=wa,\ g=wab
    \]
    constitute a $\mathbb{Z}$-basis of the center $Z(R/(L(I)+I))$, and their traces are given by the formulae
   \[
   L(\lambda)=\left\{\begin{array}{cl}
   2\lambda &\mathrm{if}\ \lambda\in\{1,w,d,e\};\\
   -w&\mathrm{if}\ \lambda=f;\\
   w+e&\mathrm{if}\ \lambda=g.
   \end{array}
   \right.
   \]
   \item The kernel of the surjective ring homomorphism $\alpha_G\colon B'G\to A'G$ is the ideal in $R/I$ generated by the elements $\epsilon$ and $w-d$.
    \end{enumerate}
    
\end{document}